\documentclass[12pt,letterpaper]{article}
\usepackage{graphicx}
\usepackage{amsmath}
\usepackage{amsfonts}
\usepackage{amsthm}
\usepackage{amssymb}
\usepackage{color}
\def\R{\mathbb{R}}
\usepackage{verbatim} 
\usepackage{hyperref}

\newtheorem{theorem}{Theorem}[section]
\newtheorem{lemma}[theorem]{Lemma}
\newtheorem{corollary}[theorem]{Corollary}

\newcommand{\be}{\begin{equation}}
\newcommand{\ee}{\end{equation}}
\newcommand{\bea}{\begin{eqnarray}}
\newcommand{\eea}{\end{eqnarray}}
\newcommand{\beas}{\begin{eqnarray*}}
\newcommand{\eeas}{\end{eqnarray*}}

\begin{document}

\title{Fission of Halving Edges Graphs}

\author{Tanya Khovanova\\MIT \and Dai Yang\\MIT}

\maketitle

\begin{abstract}
In this paper we discuss an operation on halving edges graph that we call fission. Fission replaces each point in a given configuration with a small cluster of $k$ points. The operation interacts nicely with halving edges, so we examine its properties in detail.
\end{abstract}

\section{Introduction}

Halving lines have been an interesting object of study for a long time. Given $n$ points in general position on a plane the minimum number of halving lines is $n/2$. The maximum number of halving lines is unknown. The current lower bound of $\Omega(ne^{\sqrt{\ln 4}\sqrt{\ln n}}/\ln n)$ was found by Nivasch \cite{Nivasch}, an improvement from T\'{o}th's lower bound \cite{Toth}.

The current asymptotic upper bound of $O(n^{4/3})$ was proven by Dey \cite{Dey98}. In 2006 a tighter bound for the crossing number was found \cite{PRTT}, which also improved the upper bound for the number of halving lines. In our paper \cite{KY} we further tightened the Dey's bound. This was done by studying the properties of halving edges graphs, also called underlying graphs of halving lines. 

We discussed connected components of halving edges graphs in \cite{KY2}. In this paper we continue our studies of halving edges graphs. In particular, we introduce a construction which we call fission. Fission is replacing each point in a given configuration with a small cluster of $k$ points; this operation produces elegant results in relation to halving lines.

We start in Section~\ref{sec:definitions} with supplying necessary definitions and providing examples. In Section~\ref{sec:fission} we define fission and prove some results pertaining to the behavior of newly appearing halving lines. Then, in Section~\ref{sec:chains} we discuss the behavior of chains. In the next Section~\ref{sec:plain} we define and discuss the simplest case of fission called plain fission, in which no halving lines are generated within each cluster. In Section~\ref{sec:multiplication} we note the similarities between fission and multiplication as well as lifting. In Section~\ref{sec:parallel} and Section~\ref{sec:forest} we study two more special cases of fission. The former is called parallel fission and describes fission when all clusters are sets of points on lines parallel to each other. The latter is called a fission of a 1-forest and is an interesting example of fission where the starting graph is a 1-forest. In the last Section~\ref{sec:defission} we discuss the opposite operation---defission---which is analog of division for graphs that can be created through fission.

\section{Definitions and Examples}\label{sec:definitions}

Let $n$ points be in general position in $\R^2$, where $n$ is even. A \textit{halving line} is a line through 2 of the points that splits the remaining $n-2$ points into two sets of equal size.

From our set of $n$ points, we can determine an \textit{underlying graph} of $n$ vertices, where each pair of vertices is connected by an edge iff there is a halving line through the corresponding 2 points. The underlying graph is also called the \textit{halving edges graph}.

We denote the number of vertices in graph $G$ as $|G|$.

In dealing with halving lines, we consider notions from both Euclidean geometry and graph theory. A \textit{geometric graph}, or a \textit{geograph} for short, is a pair of sets $(V,E)$, where $V$ is a set
of points on the coordinate plane, and $E$ consists of pairs of elements from
$V$. In essence, a geograph is a graph with each of its vertices assigned
to a distinct point on the plane. Many problems in the intersection of geometry and graph theory have relevance to geographs, and the study of halving lines is no different. Earlier works on geographs include \cite{Pach}.

\subsection{Examples}

\subsubsection{Four points}

Suppose we have four non-collinear points. If their convex hull is a quadrilateral, then there are two halving lines. If their convex hull is a triangle, then there are three halving lines. Both cases are shown on Figure~\ref{fig:4points}.

\begin{figure}[htbp]\label{fig:squaretriangle}
\begin{center}
\includegraphics[scale=0.5]{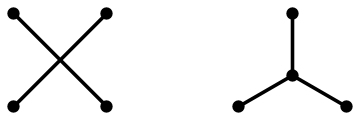}
\end{center}
  \caption{Underlying graphs for four points.}\label{fig:4points}
\end{figure}

\subsubsection{Polygon}

If all points belong to the convex hull of the point configuration, then each point lies on exactly one halving line. The number of halving lines is $n/2$, and the underlying graph is a matching graph --- a union of $n/2$ disjoint edges. The left side of Figure~\ref{fig:4points} shows an example of this configuration.

\subsubsection{Star}

If our point configuration is a regular $(n-1)$-gon and a point at its center, then the underlying graph is a star with a center of degree $n-1$. The configuration has $n-1$ halving lines. The right side of Figure~\ref{fig:4points} shows an example of this configuration.

\subsection{Chains}

Dey~\cite{Dey98} uses the notion of \textit{convex chains} to improve the upper bound on the maximum number of halving lines. Chains prove to be useful in proving other properties of the underlying graph as seen in \cite{KY}. The following construction partitions halving lines into chains:
\begin{enumerate}
\item Choose an orientation to define as ``up." The $\frac{n}{2}$ leftmost vertices are called the left half, and the rightmost vertices are called the right half.
\item Start with a vertex on the left half of the graph, and take a vertical line passing through this vertex.
\item Rotate this line clockwise until it either aligns itself with an edge, or becomes vertical again.
\item If it aligns itself with an edge in the underlying graph, define this edge to be part of the chain, and continue rotating the line about the rightmost vertex in the current chain.
\item If the line becomes vertical, we terminate the process. The set of edges in our set is defined as the chain.
\item Repeat step 2 on a different point on the left half of the underlying graph until every edge is part of a chain.
\end{enumerate}

The construction is illustrated in Figure~\ref{fig:chains}. The thickest broken path is the first chain. The next chain is thinner, and the third chain is the thinnest. Note that the chains we get are determined by which direction we choose as ``up."

\begin{figure}[htbp]
\begin{center}
\includegraphics[scale=0.4]{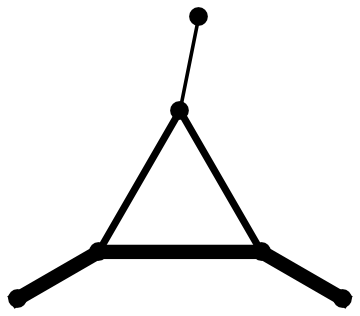}
\end{center}
  \caption{Chains.}\label{fig:chains}
\end{figure}

The following properties of chains follow immediately. Later properties on the list follow from the previous ones \cite{Dey98} and \cite{KY}:

\begin{itemize}
\item A vertex on the left half of the underlying graph is a left endpoint of a chain.
\item The process is reversible. We could start each chain from the right half and rotate the line counterclockwise instead, and obtain the same chains.
\item A vertex on the right half of the underlying graph is a right endpoint of a chain.
\item Every vertex is the endpoint of exactly one chain.
\item The number of chains is exactly $\frac{n}{2}$.
\item The degrees of the vertices are odd. Indeed, each vertex has one chain ending at it and several passing through it.
\item Every halving line is part of exactly one chain.
\end{itemize}

\subsection{Segmentarizing}

Suppose we have a set of points. Any affine transformation does not change the set of halving lines. Sometimes it is useful to picture that our points are squeezed into a long narrow rectangle. This way our points are almost on a segment. We call this procedure \textit{segmenterizing}, and we introduced it in~\cite{KY}. The Figure~\ref{fig:squeezing} shows three pictures. The first picture has six points, that we would squeeze towards the line $y=0$. The second picture shows the configuration squeezed by a factor of 10, and if we make the factor arbitrary large the points all lie very close to a segment as shown on the last picture.

\begin{figure}[htbp]
\begin{center}
\includegraphics[scale=0.7]{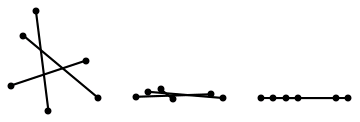}
\end{center}
  \caption{Segmenterizing.}\label{fig:squeezing}
\end{figure}

\subsection{Cross}

The following construction we call a \textit{cross} \cite{KY}. 

We form the cross as follows. We squeeze initial sets into long narrow segments. Then we intersect these segments at middle lines, so that half of the points of each segment lie on one side of all halving lines that pass through the points of the other segment (See Figure~\ref{fig:cross}).

Given two sets of points with $n_1$ and $n_2$ points respectively whose underlying geographs are $G_1$ and $G_2$, the cross is the construction of $n_1+n_2$ points on the plane whose underlying geograph has two isolated components $G_1$ and $G_2$. 

\begin{figure}[htbp]
\begin{center}
\includegraphics[scale=0.4]{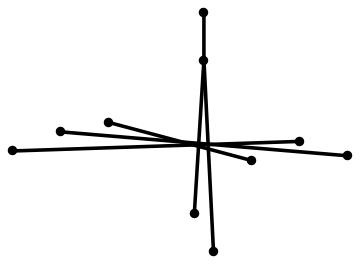}
\end{center}
  \caption{The Cross construction.}\label{fig:cross}
\end{figure}

Note that the left example in Figure~\ref{fig:4points} is a cross of two 2-paths.

\section{Fission}\label{sec:fission}

We introduce a construction that we call \textit{k-fission}. Replace each point in the graph $G$ by a small cluster of $k$ points that all are not more than $\epsilon$ away from the replaced point, for some small $\epsilon$. We call the cluster that replaces the point $v_i$ the \textit{$v_i$ cluster}. Figure~\ref{fig:2fission} shows a picture of a 2-fission of the star configuration from the right of Figure~\ref{fig:4points}.

How small should $\epsilon$ be? First, we want $\epsilon$ to be much smaller than any distance between the given points. This way different clusters do not interfere with each other. In addition, we want the lines connecting points from two different clusters not to pass through other clusters. For this purpose we introduce the notion of $\epsilon$-corridor. That means that $\epsilon$ is small enough so that if a line connects two points from the clusters of $v$ and $w$, then the other clusters are on the same side of the line if and only if the corresponding points prior to fission are on the same side of the line $vw$.

Given a segment, call all the directions that can be formed by lines connecting points that are not more than $\epsilon$ apart from the ends of the segment, its \textit{
$\epsilon$-corridor}. We assume that $\epsilon$ is so small that no two segments connecting points in our configuration $G$ have overlapping $\epsilon$-corridors.

\begin{figure}[htbp]
\begin{center}
\includegraphics[scale=0.2]{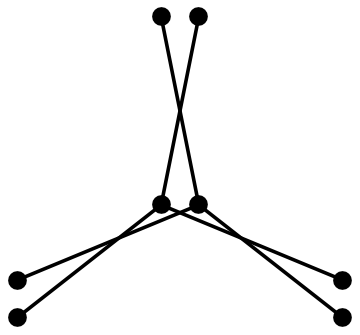}
\end{center}
  \caption{2-Fission example.}\label{fig:2fission}
\end{figure}

Given an underlying graph $G$ and its $k$-fission $H$, we know that $k|G| = |H|$. If $vw$ is an edge of $G$, then any edge of $H$ with one endpoint in the $v$ cluster and one endpoint in the $w$ cluster is said to \textit{traverse} these two clusters. 

We begin by proving the following two lemmas:

\begin{lemma}
If two points $v$ and $w$ do not form a halving edge of $G$, then there are no halving edges in $H$ traversing from the cluster of $v$ to the cluster of $w$.
\end{lemma}

\begin{proof}
Suppose the line $vw$ has $a$ points from $G$ on one side, and $n-a-2$ points on the other side, where $a$ is not $n/2$. Then a line traversing from the cluster of $v$ to the cluster of $w$ has at least $ka$ points from $H$ on one side and at least $k(n-a-2)$ points on the other side. The leftover $2k-2$ points cannot bring the balance to zero.
\end{proof}

\begin{lemma}\label{thm:traversing}
For any halving edge $vw$ of $G$, there are exactly $k$ corresponding halving edges in $H$ traversing from the cluster of $v$ to the cluster of $w$.
\end{lemma}

\begin{proof}
Consider the geograph $H'$ consisting of solely the $2k$ vertices in $H$ which are in the clusters of $v$ and $w$. Orient $H'$ so that the $k$ vertices which correspond to $v$ all lie to the left half of $H'$, and the $k$ vertices which correspond to $w$ all lie on the right half. Then since there are exactly $k$ chains in $H'$, and each chain contains one edge that traverses from the left half to the right half, there must be exactly $k$ edges between the cluster of $v$ and the cluster of $w$.

Now since $H$ is a fission of $G$, all the vertices of $H$ which belong to neither the cluster of $v$ nor $w$ must be divided exactly in half by any line of the form $v_1w_1$, where $v_1,w_1$ are vertices of $H$ such that $v_1$ is in the cluster of $v$ and $w_1$ is in the cluster of $w$. Hence, the edges that traverse from the cluster of $v$ to the cluster of $w$ are preserved upon deleting all the vertices of other clusters. But this gives precisely the geograph $H'$, so we are done.
\end{proof}

The previous proof gave us a description how traversing halving edges of a $k$-fission are arranged.

\begin{corollary}\label{thm:traversinginduced}
For any halving edge $vw$ of $G$, the $k$ corresponding traversing halving edges in $H$ are halving edges of a subgraph formed by $2k$ points of two clusters $v$ and $w$.
\end{corollary}

Intuitively, our lemmas assert that every edge of $G$ splits conveniently into $k$ traversing edges in $H$ under fission and there are no other traversing edges. However, it does not assert anything about the number of non-traversing edges of $H$, namely those whose endpoints lie in the same cluster. In Figure~\ref{fig:segment2fission} there is a 2-point configuration on the left. In the middle its 2-fission contains a halving line connecting two points within the cluster of the left vertex. On the right the 2-fission does not contain non-traversing edges.

\begin{figure}[htbp]
\begin{center}
\includegraphics[scale=0.5]{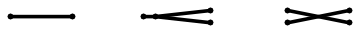}
\end{center}
  \caption{2-Path and its two different 2-fissions.}\label{fig:segment2fission}
\end{figure}

There are, however, restrictions on the orientation of such halving lines if they do appear.

\begin{theorem}\label{thm:corridor}
If there is a non-traversing halving line within a $v$ cluster, then the direction of this halving line is within the $\epsilon$-corridor of one of the halving edges of $G$ with an end point at $v$.
\end{theorem}

\begin{proof}
If the direction of a line connecting two points in a cluster does not belong to any of the $\epsilon$-corridors, then the line cannot pass through any of the other clusters. That means that the number of full clusters on each side of the line is different and this line cannot be a halving line.
\end{proof}

Theorem~\ref{thm:corridor} shows that it is easy to build fission examples such that there are no halving lines within clusters. We just need to choose points in clusters in such a way that lines connecting any two points in the cluster have directions that do not coincide with any $\epsilon$ corridors. We will call a fission such that non-traversing edges do not appear a \textit{plain} fission. 

\begin{corollary}
A geograph $H$ is a plain $k$-fission of a geograph $G$ if and only if the number of edges of $H$ is $k$ times the number of edges of $G$.
\end{corollary}

\section{Chains under Fission}\label{sec:chains}

It is natural to consider the relation between the chains of an underlying geograph, and the chains of its fission. We can use lemmas above to deduce that chains in $G$ split into $k$ chains in $H$.

\begin{theorem}
For every chain in $G$ with vertices $v_1,v_2,...,v_j$, there are exactly $k$ corresponding chains in $H$ which traverse from the cluster of $v_i$ to the cluster of $v_{i+1}$ for every $i$.
\end{theorem}

\begin{proof}
If a chain $C$ in $G$ contains edge $vw$, and a chain $D$ in $H$ contains an edge traversing from the cluster of $v$ to the cluster of $w$, then we say that $D$ \textit{overlaps} with $C$. Chain $D$ cannot overlap with more than one chain $C$: this follows from the construction of chains, in which every edge of a chain determines the next. Now note that there are exactly $k$ times as many chains in $H$ than in $G$. Therefore, we must have exactly $k$ chains in $H$ which overlap with any chain $C$ in $G$. Since every edge in $C$ corresponds to $k$ traversing edges in $H$, the $k$ chains in $H$ overlapping with $C$ must traverse all the clusters corresponding to the vertices of $C$, as desired.
\end{proof}

\section{Plain Fissions}\label{sec:plain}

We know that in plain fission there are exactly $k$ edges traversing two clusters $v$ and $w$. Moreover, from the Corollary~\ref{thm:traversinginduced} we know that the having lines traversing $v$ and $w$ are the halving lines of the graph consisting of $2k$ vertices that belong to $v$ and $w$ clusters. As there are no more halving lines within the two clusters and each vertex of a graph has to have at least one halving line we get the following lemma.

\begin{lemma}\label{thm:covering}
Every vertex in a $v$ cluster of a plain fission has exactly one halving edge connecting to a $w$ cluster if and only if $v$ and $w$ are connected in $G$.
\end{lemma}

\begin{corollary}
Every vertex in the $v$ cluster of a plain fission has the same degree $d$, which is the degree of $v$ in $G$.
\end{corollary}

We already know from Figure~\ref{fig:segment2fission} that two different $k$-fissions of the same graph can produce different halving edges graphs. Is the same true for plain fissions? It turns out that two plain fissions of the same graph can indeed produce different graphs. Suppose graph $G$ has a cycle of length 3, as for example a graph with 6 vertices in Figure~\ref{fig:graph6}.

\begin{figure}[htbp]
\begin{center}
\includegraphics[scale=0.2]{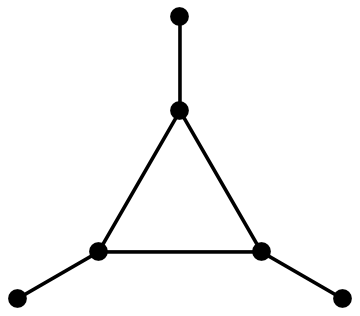}
\end{center}
  \caption{A halving edge graph with a 3-cycle.}\label{fig:graph6}
\end{figure}

This graph can produce two different resulting graphs under plain 2-fission. We will only draw the part of the graph that corresponds to the 3-cycle, since the halving edges of a fission of a subgraph do not depend on the rest of the graph, see Corollary~\ref{thm:traversinginduced}. In Figure~\ref{fig:differentfissions} the left 2-fission example consists of two 3-cycles, and the right example consists of a single 6-cycle.

\begin{figure}[htbp]
\begin{center}
\includegraphics[scale=0.3]{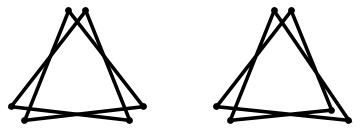}
\end{center}
  \caption{Plain 2-fissions can produce different results.}\label{fig:differentfissions}
\end{figure}

\section{Fission as Multiplication}\label{sec:multiplication}

Suppose that every cluster consists of the configurations of $k$ points. We can view this fission as multiplication of the halving edges graph by a configuration. If configuration does not have two points that generate a line within the direction of any $\epsilon$-corridor, then it is a plain fission. With such a multiplication the resulting graph has $k$ times more vertices and $k$ times more edges.

The identity operation in this multiplication is a $1$-fission: replacing every vertex by itself. We can also see that fission is transitive. We can consider an $m$-fission of a $k$-fission of a graph as a $km$-fission.

Does this multiplication depend on configuration or only on the number of points? Figure~\ref{fig:differentsameclusterfissions} shows a part of a 3-fission of a triangle subgraph of a graph. The resulting halving lines form a 9-cycle. We will see in the next section that a 3-fission of a triangle subgraph is a 3-cycle and a 6-cycle if the cluster consists of 3 nearly collinear points. This proves that the graph does depend on the configuration.

\begin{figure}[htbp]
\begin{center}
\includegraphics[scale=0.3]{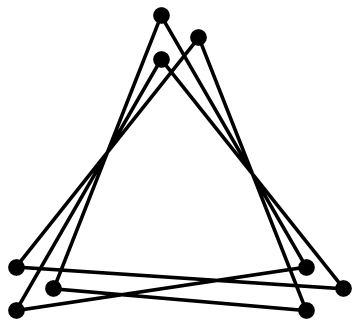}
\end{center}
  \caption{A 3-fission example of a triangle subgraph.}\label{fig:differentsameclusterfissions}
\end{figure}

Now we consider how fission affects connected components.

\begin{theorem}
If a graph $H$ is a $k$-fission of a connected graph $G$, then every connected component of $H$ passes through all the clusters. Moreover, any given connected component of $H$ has the same number of vertices in every cluster of $G$.
\end{theorem}

\begin{proof}
Take any two points of $G$ which have halving lines passing through them, say $v$ and $w$. Let $\mathcal{T}$ be a connected component of $H$. It suffices to show that $\mathcal{T}$ has the same number of vertices in the clusters of $v$ and $w$. Removing all the vertices of $G$ not belonging to these two clusters does not affect the halving lines from $v$-cluster to $w$-cluster. Call the resulting graph $R$. Orient $R$ such that all the vertices of the $v$-cluster lie on the left half.

Suppose that $\mathcal{T}_R$ is a connected component of $R$. By \cite{KY2}, $\mathcal{T}_R$ has as many vertices on the left half of $R$ as on the right half. Hence, it has as many vertices in $v$ and $w$ clusters, regardless of orientation. However, this is true for any choice of connected component $\mathcal{T}_R$ of $R$, and the intersection of $\mathcal{T}$ and $R$ is simply a union of connected components of $R$. Hence, $\mathcal{T}$ itself must have equally many vertices in the clusters of $v$ and $w$.

By the same argument, given any two clusters of $H$ which have halving lines passing through them, we can conclude that $\mathcal{T}$ has the same number of vertices in each. But $H$ is a $k$-fission of $G$, which is connected, so in fact $\mathcal{T}$ has the same number of vertices in every cluster.
\end{proof}

\begin{corollary}\label{thm:fissionConCor}
If $G$ is connected, then the number of vertices in any connected component of its $k$-fission is a multiple of $|G|$.
\end{corollary}

Note that plain fission can be viewed as graph lifting or covering. If $G$ and $H$ are two graphs, then $H$ is called a \textit{covering graph} or a \textit{lift} of $G$ if there is a surjection from the points of $H$ to the points of $V$ and the restriction of this map to any vertex of $H$ is a bijection with respect to the neighborhoods \cite{Godsil}. There is a natural surjection of $H$, that is a $k$-fission of $G$, onto $G$: every point in the $v$-cluster maps to $v$.

\begin{theorem}
The $k$-fission graph $H$ is a covering of $G$, if the $k$-fission is plain.
\end{theorem}

\begin{proof}
The proof directly follows from Lemma~\ref{thm:covering}.
\end{proof}

\section{Parallel Fission}\label{sec:parallel}

We now consider a special type of plain fission. Suppose that we replace each point by the same set of $k$ points on a line, and that the orientation of this line does not belong to any $\epsilon$-corridor. We call such a fission a \textit{parallel fission}. After properly rotating a parallel $k$-fission graph we can assume that points in every cluster form a horizontal line. Thus, we can identify the left and the right points in each cluster.

To prevent our $k$-fission graph from having collinear points, we can perturb the points very slightly to almost lie on a line within each cluster. In this manner, any two points in a cluster will be connected by a line almost in the same direction, and all the angular directions achieved by connecting two such points form a set of directions that does not overlap with any of the $\epsilon$-corridors of edges of graph $G$.

\begin{lemma}
A parallel $k$-fission of the halving edges graph $G$ can be divided into connected components, where each component contains only $m$-th left and $m$-th right points in each cluster.
\end{lemma}

\begin{proof}
If a traversing edge from $v$ to $w$ cluster starts with the $m$-th point from the left, it has to connect to the $m$-th point from the right.
\end{proof}

The lemma allows to reduce the description of the structure of the parallel $k$-fission graph to 2-fission graphs. But first we add a standard notation: $K_n$ is a complete graph with $n$ vertices.

\begin{lemma}
The parallel 2-fission of the graph $G$ is the tensor product of $G$ and $K_2$.
\end{lemma}

The proof of the lemma is straightforward after we define tensor products of two graphs that was introduced by A.~N.~ Whitehead and B.~Russell in \cite{WR}.

The \textit{tensor product} $G_1 \times G_2$ of graphs $G_1$ and $G_2$ is a graph such that the vertex set of $G_1 \times G_2$ is the Cartesian product of the vertex sets of $G_1$ and $G_2$; and any two vertices $(v,v')$ and $(w,w')$ are adjacent in $G_1 \times G_2$ if and only if $v$ is adjacent with $w$ and $v'$ is adjacent with $w'$. In particular if $G_2$ is $K_2$, then any two vertices $(v,v')$ and $(w,w')$ are adjacent if $v$ is adjacent with $w$ and $v'$ and $w'$ are two different vertices of $K_2$.

Now we present a theorem which immediately follows from the two lemmas above.

\begin{theorem}
The parallel $k$-fission of the graph $G$ is the union of $k/2$ tensor products of $G$ and $K_2$ if $k$ is even. If $k$ is odd, then parallel $k$-fission of the graph $G$ is the union of $(k-1)/2$ tensor products of $G$ and $K_2$ and one copy of $G$.
\end{theorem}

The following corollaries follow from properties of tensor products.

\begin{corollary}
The parallel $k$-fission of the graph $G$ is bipartite.
\end{corollary}

\begin{corollary}
If $G$ is bipartite, then the parallel $k$-fission of $G$ is $k$ copies of $G$.
\end{corollary}

\section{Forest Fission}\label{sec:forest}

In previous sections we considered $k$-fissions where there are no halving edges within one cluster, namely plain fissions. Now we discuss examples where there are a lot of halving edges within the cluster. But first, we introduce some definitions.

A graph is called a \textit{1-tree} or \textit{unicyclic} if it is connected and has exactly one cycle. A graph in which each connected component is a 1-tree is called a \textit{1-forest}.

It is easy to check that if graph $G$ is a 1-forest, then we can direct its edges in such a way that every vertex has outdegree 1. This directed graph is called the \textit{functional graph}. 

We define a new type of fission which we call \textit{forest} fission. As a prerequesite, our halving edges graph $G$ must be a 1-forest. Each cluster will be a configuration of points that match some given halving lines configuration $B$. We define the forest fission as follows. Replace each vertex of $G$ by a segmenterized copy of $B$ that fits in an $\epsilon$ neighborhood of the vertex. In addition, let the segmenterized vertices be aligned with the edge coming out of the original vertex in the functional graph of $G$ and the vertex divide the point in $B$ in half. 

Figure~\ref{fig:forestfission} gives an example of forest fission. We start with a smallest halving line configuration that is a forest. In this case the configuration has 6 vertices. The left picture shows the directions of each edge so that the outdegree is 1. In the middle picture we replace each vertex with a segment of two vertices, so that the segment is oriented in the direction of the outedge and the middle of the segment is located at the former vertex. The last picture shows the resulting forest fission configuration.

\begin{figure}[htbp]
\begin{center}
\includegraphics[scale=0.3]{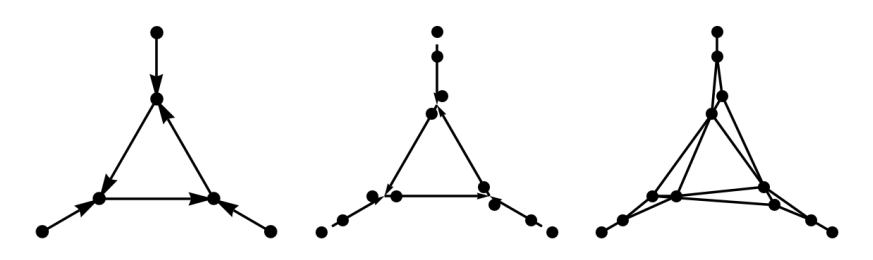}
\end{center}
  \caption{Forest Fission.}\label{fig:forestfission}
\end{figure}

\begin{lemma}
The non-traversing halving lines of the forest fission within any given cluster are exactly the halving lines of that cluster when all other clusters are removed.
\end{lemma}

\begin{proof}
Consider two points in the cluster of $v$. Suppose that in $G$, the line from $v$ was directed towards $w$. Every line passing through two points in the same cluster divides all clusters other than $v$ and $w$ in half. Also, it divides the points in the $w$-cluster in half. Therefore, it divides all points not in $v$-cluster in half. Thus, a line through two points in the $v$-cluster is the halving line of the fission graph if and only if it is a halving line in $B$, or equivalently, an non-traversing halving line in $G$.
\end{proof}

\begin{corollary}
If graph $G$ has $e_G$ edges and graph $B$ has $e_B$ edges, then the forest fission gives a set with $|G||B|$ vertices and $|B|e_G+e_Ge_B$ edges.
\end{corollary}

As $G$ is 1-forest, we know that $|G| = e_G$.

The smallest known unicyclic halving edges graph has 6 vertices and is depicted on Figure~\ref{fig:graph6}. If we use this graph as $G$ in the forest fission construction we can get from $B$ a graph with $6|B|$ vertices and $6|B|+6e_B$ edges. Suppose we start with $B$ the same as $G$ and create their forest fission graph. Then we fission the graph $G$ with the resulting graph from the previous step and continue this recursively. We will get a sequence of graphs $g(k)$ that have $6^k$ vertices and $k\cdot 6^k$ edges. That gives us another construction of an $n$-point configuration with at least $n\log n$ halving edges.

\section{Defission}\label{sec:defission}

Thanks to multiplication, we have a notion of divisibility. If $H$ is a fission of $G$, we will say that $G$ \textit{divides} $H$. The most natural follow-up questions to ask are those related to divisibility of numbers.

For instance, given a geograph, we can trivially represent it as the fission of another smaller graph:

\begin{lemma}
Every halving line graph $G$ divides the 2-path graph.
\end{lemma}

\begin{proof}
We can segmentarize $G$, divide vertices into left and right halves and move these halves away from each other, so that each half is a cluster.
\end{proof}

As we mentioned before if $G$ divides $H$, then $|G|$ divides $|H|$. The following lemma follows:

\begin{lemma}
If $G$ has $2p$ vertices, where $p$ is prime, then it only has two divisors itself and the two-path. 
\end{lemma}

As we mentioned before, fission is a transitive operation, so divisibility is transitive as well.

Are there ``prime" halving geographs? If fission is similar to multiplication, what are prime building blocks of halving edges graphs? Let us call a halving edges graph \textit{primitive} if it only divides itself and 2-path. Clearly if the geograph in question has $2p$ vertices, where $p$ is prime, then it is primitive. Are there other non-trivially primitive graphs?

\begin{lemma}
The graph that is a cross of a 6-star and a 2-path is primitive.
\end{lemma}

\begin{proof}
The only possibility for this to be a non-trivial fission is for it to be a 2-fission of a connected graph with 4 vertices. Then by Corollary~\ref{thm:fissionConCor} each connected component has to have at least 4 vertices. Contradiction.
\end{proof}

We can imitate the proof above to easily produce other primitive underlying geographs which do not have $2p$ vertices, with several connected components. However, it is unclear whether there exists a primitive connected halving edges graph; this poses an interesting object of study for future exploration into fission and divisibility.

\section{Acknowledgements}

We are grateful to Professor Jacob Fox for helpful discussions. The second author was supported by UROP.

\end{document}